\documentclass[11pt,a4paper]{article}
\usepackage[top=2.5cm,bottom=2.5cm,left=2.2cm,right=2.2cm]{geometry}
\usepackage[T1]{fontenc}
\usepackage[utf8]{inputenc}
\usepackage{color}
\usepackage{graphicx}
\usepackage{amsfonts}
\usepackage{extarrows}
\usepackage{amsmath,amsthm,amssymb,color}
\usepackage{hyperref}
\usepackage{eepic}
\usepackage{lineno}
\usepackage{enumerate}	
\usepackage[utf8]{inputenc}
\usepackage{paralist}
\usepackage{cite}
\usepackage{algorithm}
\usepackage{algorithmicx}
\usepackage{algpseudocode}
\usepackage{authblk}
\usepackage{mathtools}
\usepackage{pifont}
\usepackage[numbers,sort&compress]{natbib}
\usepackage{comment}
\usepackage{ulem}

\usepackage{tikz}
\usetikzlibrary{positioning}

\hypersetup
{
	colorlinks=true,
	linkcolor=blue,
	filecolor=blue,
	urlcolor=blue,
	citecolor=cyan,
}

\usepackage{etoolbox}

\newtheorem{theorem}{Theorem}

\newtheorem{lemma}[theorem]{Lemma}
\newtheorem{nota}{Notation}

\newtheorem{conjecture}[theorem]{Conjecture}

\theoremstyle{definition}

\AtBeginEnvironment{proof}{\setcounter{case}{0}}
{\setlength{\leftmargini}{1.5\parindent}
	\begin{itemize}
		\setlength{\itemsep}{-1.1mm}}
	{\end{itemize}}

\newcommand{\ve}{\varepsilon}
\begin{document}
	\title{\bf Further Results on the Maximum Number of Stars in Graphs with Forbidden Properties
		\thanks{This research is supported by National Key R\&D Program of China under grant number 2024YFA1013900, NSFC under grant number 12471327,  JSPS KAKENHI Grant Number 25KF0036, the NSF of Hubei Province Grant Number 2025AFB309,  the China Postdoctoral Science Foundation  Grant Number 2025M773113, the Fundamental Research Funds for the Central Universities, Central China Normal University Grant Number CCNU24XJ026.}}
	
	\author[1]{\bf Yuxuan Liu \footnote{Email: hlioukyu@gmail.com.}}
	\author[2]{\bf Jia-bao Yang \footnote{Email: jbyang1215@163.com.}}
	\author[3,4]{\bf Leilei Zhang \footnote{Corresponding author. Email: mathdzhang@163.com.}}
	
	\affil[1]{\footnotesize Department of Mathematics, Kyushu University, Motooka 744, Nishi-ku, Fukuoka, 819-0395, Japan}
	\affil[2]{\footnotesize School of Mathematics, Nanjing University, Nanjing, 210093, China}
	\affil[3]{\footnotesize School of Mathematics and Statistics, Central China Normal University, Wuhan 430079, China}
	\affil[4]{\footnotesize Faculty of Environment and Information Sciences, Yokohama National University, Yokohama 240-8501, Japan}
	\date{}
	\maketitle
	\begin{abstract}
		A graph $G$ is called $k$-edge-hamiltonian if every linear forest (i.e., a disjoint union of paths) with at most $k$ edges is contained in a Hamilton cycle of $G$. In 2018, F\"uredi, Kostochka and Luo determined the maximum number of $t$-stars in nonhamiltonian graphs, thereby extending an earlier result of Erd\H{o}s. Recently, Berikkyzy, Hogenson, Kirsch and McDonald extended this line of research by determining the maximum number of $t$-stars in graphs that are not $k$-edge-hamiltonian, as well as in graphs failing to satisfy related properties such as traceability, hamiltonian-connectedness and $k$-hamiltonicity. For sufficiently large $t$, they also characterized the extremal graphs, while for smaller values of $t$, they proposed a conjecture. In this paper, we investigate this conjecture. We show that the conjecture fails at the critical value and further establish a threshold-type result describing the behavior of the extremal graphs when $t$ is close to this critical value.
		
		\smallskip
		
		\smallskip
		
	\end{abstract}
	
	\section{ Introduction}
	
	We adopt standard notation and terminology in graph theory; undefined terms follow Bondy and Murty \cite{BM}. Throughout, all graphs are assumed to be simple, finite and undirected.  For two vertices $u$ and $v$, we use the symbol $u\leftrightarrow v$ to mean that $u$ and $v$ are {\it adjacent} and use $u\nleftrightarrow v$ to mean that $u$ and $v$ are {\it nonadjacent}.  Let $\overline{G}$ denote the complement of $G.$  For two graphs $G$ and $H$, $G\vee H$ denotes the join of $G$ and $H$, which is obtained from the disjoint union $G\cup H$ by adding all edges	joining every vertex of $G$ to every vertex of $H$. A $t$-star $S_t$ is a star consisting of a central vertex and $t$ leaves. An {\it $s$-clique} is a clique of cardinality $s.$ Throughout the paper, $K_s$ denotes the complete graph on $s$ vertices, and $I_s$ denotes an independent set of size $s$.   For graphs we will use equality up to isomorphism, so $G_1=G_2$ means that $G_1$ and $G_2$ are isomorphic. 
	
	A {\it Hamilton cycle} (resp. {\it Hamilton path}) in a graph is a cycle (resp. path) that passes through all vertices.  A graph containing a Hamilton cycle is called a {\it hamiltonian graph}, and a graph containing a Hamilton path is called a {\it traceable graph}. A graph is {\it hamiltonian-connected} if and only if there exists a Hamilton path between every pair of distinct vertices. We begin by introducing an important family of graphs.
	
	\begin{nota}\label{nota1}
		Let $n,i,\ell\in\mathbb{Z}$ with $-1\le\ell\leq n-3$ and $1\leq i\leq \tfrac{n-1-\ell}{2}$. Define $G(n,\ell,i)$ to be the graph $K_{i+\ell}\vee( K_{n-2i-\ell} \cup I_i)$. Denote by $g_t(n,\ell,i)$ the number of $t$-stars in $G(n,\ell,i);$ more precisely,
		$$
		g_t(n,\ell,i)=i \binom{i+\ell}{t}+(n-2i-\ell)\binom{n-i-1}{t}+(i+\ell)\binom{n-1}{t}.
		$$
	\end{nota}
	
	For convenience, we denote by $g(n,\ell,i)$ the number of edges of the graph $G(n,\ell,i)$. In this family of graphs, we will focus on two extremal choices of the parameter $i$.  First, let $i_0=\left\lfloor (n-1-\ell)/2\right\rfloor$, which corresponds to the largest possible value of $i$. In addition, for each nonnegative integer $d$, we define $i_d=\max\{1,\, d-\ell\}$. It is straightforward to verify that $i_d \le i_0$.  Note that $G(n,-1,i)$ is nontraceable, while $G(n,0,i)$ and $G(n,1,i)$ are nonhamiltonian and nonhamiltonian-connected, respectively. One of the foundational results in the study of nonhamiltonian graphs is the following celebrated theorem of Erd\H{o}s.

	\begin{theorem}{\rm (Erd\H{o}s \cite{2})}
		Let G be an $n$-vertex graph with minimum degree $\delta(G )\geq d$, where $1\leq d\leq \lfloor(n-1)/2\rfloor$. If $G$ is not hamiltonian, then $e (G ) \leq {\rm max } \{ g (n,0,i_d ), g (n,0, i_0 ) \} $.
	\end{theorem}  
	
	Extremal problems on $t$-stars can be viewed as natural generalizations of classical edge extremal problems. Extending earlier work of Gerbner~\cite{G}, Ma, Hou and Yin~\cite{MHY} determined the maximum number of $t$-stars in $\{M_{s+1},K_{k+1}\}$-free graphs, where $M_{s+1}$ denotes a matching of size $s+1$. Recently, Liu and Ren \cite{LR} determined the maximum number of $t$-stars in graphs forbidding both a matching and an arbitrary graph. The maximum number of $t$-stars has also been studied for broom-free \cite{G2}, $\mathcal{C}_{\ge k}$-free \cite{GSTZ}, $C_4$-free \cite{G3}, $S_r$-free \cite{cvk} and  $P_k$-free \cite{GSTZ}, where $\mathcal{C}_{\ge k}$ denotes cycles of length at least $k$, and $\mathcal{L}_{\ge k}$ denotes linear forests with at least $k$ edges. Additional results on $t$-stars can be found in~\cite{HQ,GP,ALNS}. 
	
	Furthermore, F\"{u}redi, Kostochka, and Luo \cite{4} showed that the extremal graphs $G(n,0,i_0)$ and $G(n,0,i_d)$ not only maximize the number of edges among nonhamiltonian graphs with $n$ vertices and minimum degree at least $d$, but also maximize the number of $t$-stars.
	
	\begin{theorem} {\rm (F\"{u}redi, Kostochka, Luo \cite{3})}\label{thm1}
		Let $G$ be an $n$-vertex graph with minimum degree $\delta(G)\geq d$, where $1\leq d\leq \lfloor\frac{n-1}{2}\rfloor$, and let  $t \in \{ 1,\ldots, n-1 \}$. If $G$ is not hamiltonian, then $s_t (G ) \leq {\rm max } \{ g_t (n,0,i_d ) ,g_t (n,0,i_0 ) \}$.
	\end{theorem}  
	
	Beyond these classical notions, various strengthened and robustness-based versions of hamiltonicity have been studied. One approach is to measure how strongly hamiltonian a graph is by means of edge-based extensions. In particular, an $n$-vertex graph is called {\it $k$-edge hamiltonian} if every linear forest (that is, a disjoint union of paths) with at most $k$ edges can be extended to a Hamilton cycle of the graph. Another approach focuses on the stability of hamiltonicity under vertex deletions. An $n$-vertex graph is said to be {\it $k$-hamiltonian } if the removal of any set of at most $k$ vertices leaves a hamiltonian graph. Note that both a $0$-edge-hamiltonian graph and a $0$-hamiltonian graph are hamiltonian graphs. As an extension of Theorem \ref{thm1}, F\"{u}redi, Kostochka and Luo established the following result.
	
	\begin{theorem}[F\"{u}redi, Kostochka, Luo \cite{4}]
			Let $G$ be an $n$-vertex graph with minimum degree $\delta(G)\geq d$, where $k+1\leq d\leq \lfloor\frac{n+k-1}{2}\rfloor$.  If $G$ is not $k$-edge hamiltonian, then $e (G ) \leq {\rm max } \{ g(n,k,i_d ), g (n,k,i_0 ) \}$.
	\end{theorem}  
	
	Motivated by the results of F\"{u}redi, Kostochka, and Luo~\cite{3,4}, Berikkyzy, Hogenson, Kirsch, and McDonald~\cite{1} generalized Theorem \ref{thm1} by replacing hamiltonicity with related properties such as traceability, hamiltonian-connected, $k$-edge hamiltonian, and $k$-hamiltonian. As an immediate consequence, they obtained the following result.
	
	\begin{theorem}[Berikkyzy, Hogenson, Kirsch, McDonald~\cite{1}]\label{thm2}
		Let $G$ be an $n$-vertex graph, and let $t \in \{1,2,\ldots,n-1\}$. Suppose that at least one of the following holds:
		\begin{itemize}
			\item[{\rm $(H_1)$}] $G$ is not hamiltonian;
			\item[{\rm $(H_2)$}] $G$ is not traceable;
			\item[{\rm $(H_3)$}] $G$ is not hamiltonian-connected;
			\item[{\rm $(H_4)$}] $G$ is not $k$-edge hamiltonian for some integer $1 \le k \le n-3$; or
			\item[{\rm $(H_5)$}] $G$ is not $k$-hamiltonian for some integer $1 \le k \le n-3$.
		\end{itemize}
		Let $\ell = 0, -1, 1, k, k$ correspond to the cases $(H_1)$-$(H_5)$, respectively. Then
		\[
		s_t(G) \le \max\{ g_t(n,\ell,1),\, g_t(n,\ell,i_0)\},
		\]
		and this bound is tight.
	\end{theorem}
	
	It is then natural to ask for a characterization of the extremal graphs. Berikkyzy, Hogenson, Kirsch, and McDonald showed that, once the larger of $g_t(n,\ell,1)$ and $g_t(n,\ell,i_0)$ is determined, the extremal family can be described completely (Theorem 7 in \cite{1}). For the case $\ell=0$, they proved the following theorem.
	
	\begin{theorem}[Berikkyzy, Hogenson, Kirsch, McDonald~\cite{1}]
		
		In Theorem~\ref{thm2}, for $\ell=0$, the following holds.
		
		\begin{itemize}
			
			\item If $t \ge \frac{n+1}{2}$, then $g_t(n,0,1)\le g_t(n,0,i_0).$
			
			\item If $t < \frac{n+1}{2}$, then $g_t(n,0,i_0) \le g_t(n,0,1).$
			
		\end{itemize}
		
		Moreover, both inequalities are strict for $n \ge 6$.
		
	\end{theorem}
	
	For the general case, they proved the following theorem.
	
	\begin{theorem}[Berikkyzy, Hogenson, Kirsch, McDonald~\cite{1}]
		In Theorem~\ref{thm2}, for all $\ell$, if $t \ge \frac{n+\ell+1}{2}$, then $g_t(n,\ell,1)\le g_t(n,\ell,i_0),$ and the inequality is strict for $0\leq \ell \leq n-5$.
	\end{theorem}
	
	For the case $t\leq \left\lfloor\frac{n+\ell}{2}\right\rfloor$, they proposed the following conjecture.
	
	\begin{conjecture}[Berikkyzy, Hogenson, Kirsch, McDonald \cite{1}]\label{con1}
		In Theorem \ref{thm2}, if $t\leq \lfloor \frac{n+\ell}{2}\rfloor$, then $g_t(n,\ell,1)\ge g_t(n,\ell,i_0)$  for sufficiently large $n$.
	\end{conjecture}
	
	In this paper, we investigate Conjecture~\ref{con1}. We first show that the conjecture fails at the critical value $t=\lfloor(n+\ell)/2\rfloor$. Furthermore, we identify a threshold phenomenon around this point: when $t$ is sufficiently close to $\lfloor(n+\ell)/2\rfloor$, we establish the following theorem, which describes the behavior of the extremal graphs in this regime.  It is easy to verify that Conjecture \ref{con1}  holds when $t=2$. Thus, throughout the rest of the proof, we may assume that $t\ge 3$. 
	
	\begin{theorem}\label{mainthm}
		Let $0<\varepsilon<\tfrac12$. If $3\le t\le \frac{n+\ell}{2+\varepsilon}$ and
		$n\ge \max\!\left\{36\left( \frac{2+\varepsilon}{\varepsilon}\right)^2,\ \ell^2\right\}$,
		then $g_t(n,\ell,1)\ge g_t(n,\ell,i_0)$.
	\end{theorem}

	\section{ Proof of Theorem \ref{mainthm}}
	We first recall several relevant definitions and results from \cite{1}. Define
	$$
	f(n,\ell,t)=g_t(n,\ell,1)-g_t(n,\ell,i_0),
	$$
	where $i_0=\lfloor \frac{n-\ell-1}{2} \rfloor$ is equal to $\frac{n-\ell}{2}-1$ when $n$ and $\ell$ have the same parity and $i_0=\frac{n-\ell-1}{2}$ if they do not. More precisely, if $n$ and $\ell$ do not have the same parity, 
	$$
	f(n,\ell,t)= {\ell+1\choose t}
	+\bigg(\frac{n-\ell-1}{2}-t+\frac{\ell+1}{n-1}t\bigg){n-1\choose t}-\frac{n-\ell+1}{2}{\frac{n+\ell-1}{2}\choose t},
	$$
	and if they do,
	$$
	f(n,\ell,t)={\ell+1\choose t}+\bigg(\frac{n-\ell}{2}-t+\frac{\ell+1}{n-1}t\bigg){n-1\choose t}
	-\bigg(\frac{n-\ell}{2}-t+1+\frac{2(\ell+1)}{n+\ell}t\bigg){\frac{n+\ell}{2}\choose t}.
	$$
	
	One can check that when $\ell=0$, the formulas above reduce to the corresponding results in \cite{1}. Our main task is to establish that $f(n,\ell,t) \geq  0.$ 
	
	\bigskip
	
	We begin by considering the critical case $t=\left\lfloor  ( n+\ell)/ 2 \right\rfloor,$ and split the analysis according to the parity of $n$ and $\ell$. Suppose first that $n$ and $\ell$ have different parity. In this case, we have $t=( n+\ell-1)/ 2.$ Hence
	
	\begin{align*}
		f(n,\ell,\frac{n+\ell-1}{2})&= 
		0+\bigg(-\ell+\frac{\ell+1}{n-1}\cdotp\frac{n+\ell-1}{2}\bigg){n-1\choose \frac{n+\ell-1}{2}}-\frac{n-\ell+1}{2}\\
		&=\frac{\ell^2+2\ell-1-(\ell-1)n}{2(n-1)}{n-1\choose \frac{n+\ell-1}{2}}-\frac{n-\ell+1}{2}\\
		&=\frac{(n-1)-\ell(n-2-\ell)}{2(n-1)}{n-1\choose \frac{n+\ell-1}{2}}-\frac{n-\ell+1}{2}.
	\end{align*}
	Note that when $2\leq \ell\leq n-3$, we have $f(n,\ell,\frac{n+\ell-1}{2})<0.$ 
	
	Next suppose that $n$ and $\ell$ have the same parity.  In this case, we have $t=( n+\ell)/ 2$. Hence
	\begin{align*}
		f(n,\ell,\frac{n+\ell}{2})&= 0+\bigg(-\ell+\frac{\ell+1}{n-1}\cdotp\frac{n+\ell}{2}\bigg){n-1\choose \frac{n+\ell}{2}}
		-\bigg(-\ell+1+\frac{2(\ell+1)}{n+\ell}\frac{n+\ell}{2}\bigg)\\
		&=\bigg(-\ell+\frac{\ell+1}{n-1}\cdotp\frac{n+\ell}{2}\bigg){n-1\choose \frac{n+\ell}{2}}-2\\
		&=\frac{n-\ell(n-3-\ell)}{2(n-1)}{n-1\choose \frac{n+\ell}{2}}-2.
	\end{align*}
	Now observe that when $\ell=n-4$, we have $f\!\left(n,\ell,\frac{n+\ell}{2}\right)=0.$ If $ 2\leq \ell<n-4$, then
	$f\!\left(n,\ell,\frac{n+\ell}{2}\right)<0.$
	
	\bigskip
	
	We now proceed to prove our main theorem. The following elementary estimate will be used repeatedly in the proof.
	
	\begin{lemma}\label{lemLog}
		For $0<a<b$ and $m\geq0$, we have
		$$\log\big(1+\frac{a}{b}\big)+
		\log\big(1+\frac{a}{b+1}\big)+\cdot\cdot\cdot
		+\log\big(1+\frac{a}{b+m}\big)
		\geq(m+1)\frac{a}{a+b+m}.$$
	\end{lemma}
	
	\begin{proof}
		From the inequality $\log(1+x)\geq\frac{x}{1+x}=1-\frac{1}{1+x}$ for $x\geq0$, we deduce that
		
		\begin{align*}
			&\log(1+\frac{a}{b})+\log(1+\frac{a}{b+1})+\ldots+\log(1+\frac{a}{b+m})\\
			\geq& (1-\frac{b}{a+b})+(1-\frac{b}{a+b})+\ldots+(1-\frac{a}{b+m})\\
			=&m+1-\big(\frac{b}{a+b}+\frac{b+1}{a+b+1}+\ldots+\frac{b+m}{a+b+m}\big)\\
			>&m+1-(m+1)\frac{b+m}{a+b+m}\\
			=&(m+1)\frac{a}{a+b+m}.
		\end{align*}
		This completes the proof.
	\end{proof}

	\begin{proof}
		To prove Theorem \ref{mainthm}, we shall prove a slightly stronger statement: the second term in the expression of $f(n,\ell,t)$ is larger than  the third term, regardless of the parity of $n$ and $\ell$. To be more precise, when $n$ and $\ell$ do not have the same parity, we shall show that
		\begin{equation}\label{diffParity}
			\bigg(\frac{n-\ell-1}{2}-t+\frac{\ell+1}{n-1}t\bigg){n-1\choose t}
			\geq\frac{n-\ell+1}{2}{\frac{n+\ell-1}{2}\choose t}
		\end{equation}
		holds, and when $n$ and $\ell$ have the same parity, we shall show that
		\begin{equation}\label{sameParity}
			\bigg(
			\frac{n-\ell}{2}-t+\frac{\ell+1}{n-1}t\bigg){n-1\choose t}
			\geq\bigg(
			\frac{n-\ell}{2}-t+1+\frac{2(\ell+1)}{n+\ell}t
			\bigg){\frac{n+\ell}{2}\choose t}.
		\end{equation}
		
		We divide the proof of Theorem \ref{mainthm} into the following two cases, according to the range of $t$.
		\vskip 3mm
		{\bf Case 1.} $4\log(\frac{3}{2}+\frac{3}{\ve})\leq t\leq \frac{n+\ell}{2+\ve}$.
		\vskip 3mm
		Observe that ${m+1\choose i}=\frac{m+1}{m+1-i}{m\choose i}$. Using this identity iteratively, we see that for $n,\ell$ with different parity, 
		$$
		{n-1 \choose t}
		=\frac{(n-1)(n-2)\cdot\cdot\cdot\frac{n+\ell+1}{2}}{(n-1-t)(n-2-t)\cdot\cdot\cdot(\frac{n+\ell+1}{2}-t)}
		{\frac{n+\ell-1}{2}\choose t}.
		$$
		Let
		$$L:=\frac{(n-1)(n-2)\cdot\cdot\cdot\frac{n+\ell+1}{2}}
		{(n-1-t)(n-2-t)\cdot\cdot\cdot(\frac{n+\ell+1}{2}-t)};$$
		then
		$$
		\log L=\log(1+\frac{t}{\frac{n+\ell+1}{2}-t})
		+\cdot\cdot\cdot+\log(1+\frac{t}{n-1-t}).
		$$
		By Lemma \ref{lemLog}, with $a=t,b=\frac{n+\ell+1}{2}-t,m=\frac{n-\ell-1}{2}-1$, we have
		$$
		L\geq e^{\frac{(n-\ell-1)t}{2(n-1)}}
		\geq e^{\frac{n-\sqrt{n}-1}{n-1}\cdot\frac{t}{2}}
		\geq e^{\frac{1}{2}\cdot2\log(\frac{3}{2}+\frac{3}{\ve})}
		=\frac{3}{2}+\frac{3}{\ve},
		$$
		provided that $\ell\leq\sqrt{n}$ and $t\geq4\log(\frac{3}{2}+\frac{3}{\ve})$ (note that $n\geq9$ guarantees $\frac{n-\sqrt{n}-1}{n-1}>\frac{1}{2}$).
		
		Therefore, to prove the inequality (\ref{diffParity}), it is sufficient to prove
		$$
		(\frac{3}{2}+\frac{3}{\ve})\big(\frac{n-\ell-1}{2}-t+\frac{\ell+1}{n-1}t\big)
		\geq \frac{n-\ell+1}{2}.
		$$
		Since  $t\leq \frac{n+\ell}{2+\ve}$ and $\ell\leq \sqrt{n},$ we have
		\begin{align*}
		\big(\frac{n-\ell-1}{2}-t+\frac{\ell+1}{n-1}t\big)
			\geq (\frac{n-\sqrt{n}-1}{2}-\frac{n+\sqrt{n}}{2+\ve}).
		\end{align*}
		By using $\frac{\sqrt{n}+1}{2}+\frac{\sqrt{n}}{2+\ve}<\sqrt{n}$ as long as $n>(\frac{2+\ve}{\ve})^2$, we can obtain
		
		\begin{align}\label{eq3}
		(\frac{n-\sqrt{n}-1}{2}-\frac{n+\sqrt{n}}{2+\ve})\ge
		(\frac{\ve}{2(2+\ve)}n-\sqrt{n}).
	\end{align}
Therefore, we have
		  \begin{align*}
			\left(\frac{3}{2}+\frac{3}{\ve}\right) \left(\frac{n-\ell-1}{2}-t+\frac{\ell+1}{n-1}t\right)
			 &\ge	\left(\frac{3}{2}+\frac{3}{\ve}\right) \left(\frac{\ve}{2(2+\ve)}n-\sqrt{n}\right) \\
			&=\frac{3}{4}n-\left(\frac{3}{2}+\frac{3}{\ve}\right)\sqrt{n}.
		\end{align*}
		 
		Now, it remains to show
		$$\sqrt{n}\left(\frac{3}{4}\sqrt{n}-(\frac{3}{2}+\frac{3}{\ve})\right)>\frac{1}{2} n>\frac{n-\ell+1}{2},$$
		which follows readily from the assumption $n\geq36(\frac{2+\ve}{\ve})^2$.

		The proof of inequality (\ref{sameParity}) is exactly the same, but we still present it here for completeness. Assume that $n$ and $\ell$ have the same parity. We need to show
		$$
		\big(\frac{n-\ell}{2}-t+\frac{\ell+1}{n-1}t\big){n-1\choose t}
		\geq\big(\frac{n-\ell}{2}-t+1+\frac{2(\ell+1)}{n+\ell}t\big){\frac{n+\ell}{2}\choose t}.
		$$
		First note that
		$$
		{n-1\choose t}=\frac{n-1}{n-1-t}\ldots\frac{\frac{n+\ell}{2}+1}{\frac{n+\ell}{2}+1-t}{\frac{n+\ell}{2}\choose t}.
		$$
		Write
		$$
		Z:=\frac{n-1}{n-1-t}\ldots\frac{\frac{n+\ell}{2}+1}{\frac{n+\ell}{2}+1-t}.
		$$
		Then, by Lemma \ref{lemLog}, with $a=t, b=\frac{n+\ell}{2}+1-t,m=\frac{n-\ell-2}{2}$, we have
		$$
		\log Z\geq(\frac{n-\ell}{2}-1)\frac{t}{n-1}.
		$$
		Hence,
		$$
		Z\geq e^{\frac{n-\ell-2}{n-1}\cdot\frac{t}{2}}\geq e^{\frac{1}{2}\cdot2\log(\frac{3}{2}+\frac{3}{\ve})}
		=\frac{3}{2}+\frac{3}{\ve}.
		$$
		Therefore, inequality (\ref{sameParity}) boils down to
		$$
		Z\big(\frac{n-\ell}{2}-t+\frac{\ell+1}{n-1}t\big)\geq\frac{n-\ell}{2}-t+1+\frac{2(\ell+1)}{n+\ell}t.
		$$
		By (\ref{eq3}), we have
		\begin{align*}
			Z\big(\frac{n-\ell}{2}-t+\frac{\ell+1}{n-1}t\big) \geq& \big(\frac{3}{2}+\frac{3}{\ve}\big)\big(\frac{\ve}{2(2+\ve)}n-\sqrt{n}\big)\\
			=&\frac{3}{4}n-\frac{3(2+\ve)}{2\ve}\sqrt{n}.
		\end{align*}
		Since $n\geq36(\frac{2+\ve}{\ve})^2$, we have $\frac{n}{4}\geq\frac{3(2+\ve)}{2\ve}\sqrt{n}$. Thus,
		$$
		\frac{3}{4}n-\frac{3(2+\ve)}{2\ve}\sqrt{n}=\frac{n}{2}+\frac{n}{4}-\frac{3(2+\ve)}{2\ve}\sqrt{n}\geq \frac{n}{2}\ge \frac{n-\ell}{2}-t+1+\frac{2(\ell+1)}{n+\ell}t.
		$$
		This completes the proof of Case 1.
		
		\vskip 3mm
		
		{\bf Case 2.} $3\leq t\leq 4\log(\frac{3}{2}+\frac{3}{\ve})$.
		\vskip 3mm
		Once $\ve\in(0,\frac{1}{2})$ is specified, the assumption $n\geq36(\frac{2+\ve}{\ve})^2$ implies $$\frac{n+\ell}{3}\geq12(\frac{2+\ve}{\ve})^2>\frac{48}{\ve^2}>\frac{48}{\ve}.$$
		On the other hand,
		$$
		4\log(\frac{3}{2}+\frac{3}{\ve})=4\log(\frac{3}{2}(1+\frac{2}{\ve}))
		=4(\log\frac{3}{2}+\log(1+\frac{2}{\ve}))
		\leq4(\log\frac{3}{2}+\frac{2}{\ve})<4+\frac{8}{\ve}<\frac{48}{\ve}.$$
		We have proved
		$$
		\frac{n+\ell}{3}\geq4\log(\frac{3}{2}+\frac{3}{\ve}).
		$$
		
		Therefore, in the following, we will prove that inequalities (1) and (2) hold for $3\le t\le \frac{n+\ell}{3}$. Throughout the following proof, we use the fact that $n > 600$, which follows from $n \ge 36\left(\frac{2+\varepsilon}{\varepsilon}\right)^2$ holds for every $0 < \varepsilon < \tfrac{1}{2}$, which implies $n > 600$.
		
		First assume that $n,\ell$ have different parity. Recall that
		$$
		L:=\frac{(n-1)(n-2)\cdot\cdot\cdot\frac{n+\ell+1}{2}}
		{(n-1-t)(n-2-t)\cdot\cdot\cdot(\frac{n+\ell+1}{2}-t)}.
		$$
		By Lemma \ref{lemLog} and the assumptions $\ell\leq\sqrt{n}$, $t\geq3$ and $n\geq600>256$, we have
		$$
		L\geq e^{\frac{(n-\ell-1)t}{2(n-1)}}
		\geq e^{\frac{n-\sqrt{n}-1}{n-1}\cdot\frac{t}{2}}
		\geq e^{\frac{256-16-1}{256-1}\cdot\frac{3}{2}}>e^{1.4}>4.
		$$
		Therefore, to prove the inequality (\ref{diffParity}), it is sufficient to prove
		$$4\big(\frac{n-\ell-1}{2}-t+\frac{\ell+1}{n-1}t\big)\geq \frac{n-\ell+1}{2}.$$
		Note that
		$$4\big(\frac{n-\ell-1}{2}-t+\frac{\ell+1}{n-1}t\big)>4(\frac{n-\sqrt{n}-1}{2}-\frac{n+\sqrt{n}}{3})>4(\frac{n}{6}-\sqrt{n})$$ and $\frac{n-\ell+1}{2}\leq\frac{n}{2}$, so it remains to show $\frac{2}{3}\sqrt{n}-4\geq\frac{1}{2}\sqrt{n}$, which follows from $n\geq 600>24^2$.
		
		Now suppose that $n$ and $\ell$ have the same parity and we need to prove inequality (\ref{sameParity}). We still write $$Z:=\frac{n-1}{n-1-t}\ldots\frac{\frac{n+\ell}{2}+1}{\frac{n+\ell}{2}+1-t};$$
		then, under the assumptions,
		$$Z\geq e^{\frac{n-\sqrt{n}-2}{n-1}\cdot\frac{t}{2}}>e^{1.4}>4.$$
		Therefore, inequality (\ref{sameParity}) reduces to
		$$
		Z\big(\frac{n-\ell}{2}-t+\frac{\ell+1}{n-1}t\big)\geq\frac{n-\ell}{2}-t+1+\frac{2(\ell+1)}{n+\ell}t.
		$$
		
		Since
		$$\frac{n-\ell}{2}-t+1+\frac{2(\ell+1)}{n+\ell}t\leq\frac{n}{2}+1$$
		and
		$$
		4\big(\frac{n-\ell}{2}-\frac{n+\ell}{3}\big)=\frac{2}{3}n-\frac{10}{3}\ell\ge\frac{2}{3}n-4\sqrt{n},
		$$

		we need only to show that $\frac{n}{6}\geq4\sqrt{n}+1$, which follows readily from $n\geq600$.
		
	\end{proof}	
	
	\section*{Declaration}
	
	\noindent$\textbf{Conflict~of~interest}$
	The authors declare that they have no known competing financial interests or personal relationships that could have appeared to influence the work reported in this paper.
	
	\noindent$\textbf{Data~availability}$
	Data sharing not applicable to this paper as no datasets were generated or analyzed during the current study.

\end{document}